\documentclass{article}

\usepackage{url}
\usepackage{amsthm}
\usepackage{subeqnarray,graphicx}

\usepackage{color}

\newcommand{\RR}{{\rm I\!R}}
\newcommand{\PP}{{\rm I\!P}}
\newcommand{\NN}{{\rm I\!N}}
\newcommand{\ZZ}{{\rm Z\!\!Z}}
\newcommand{\spl}{{\raisebox{0.3ex}{\tiny +}}}
\newcommand{\sml}{{\raisebox{0.3ex}{\tiny -}}}

\newtheorem{theorem}{Theorem}
\newtheorem{remark}{Remark}
\newtheorem{corollary}{Corollary}

\begin{document}
\title{Approximation with Conditionally Positive Definite Kernels on Deficient Sets}
\author{Oleg Davydov\thanks{University of Giessen, Germany,
\texttt{oleg.davydov@math.uni-giessen.de}}}

\maketitle              %

\begin{abstract}
Interpolation and approximation of functionals with conditionally positive definite kernels
is considered on sets of centers that are not determining for polynomials. It is shown that polynomial consistency is
sufficient in order to define kernel-based numerical approximation of the functional with usual properties of optimal
recovery. Application examples include generation of sparse kernel-based numerical differentiation formulas for the 
Laplacian on a grid and accurate approximation of a function on an ellipse.

%\keywords{Conditionally positive definite kernels, numerical differentiation, optimal recovery, saddle point problem}
\end{abstract}

\section{Introduction}
Let $\Omega$ be a set and $P$ a finite dimensional space of functions on $\Omega$. A function $K:\Omega\times\Omega\to\RR$ is
said to be a \emph{conditionally positive definite kernel with respect to $P$} if for any finite set 
$X=\{x_1,\ldots,x_n\}\subset\Omega$ the quadratic form $\sum_{i,j=1}^n c_ic_jK(x_i,x_j)$ is positive
for all $c\in\RR^n\setminus\{0\}$ such that $\sum_{i=1}^n c_ip(x_i)=0$ for all $p\in P$ \cite{Wendland}.

Given data $(x_j,f_j)$, $j=1,\ldots,n$, with $x_j\in \Omega$, $f_j\in\RR$, a sum of the form
\begin{equation}\label{sigma}
\sigma(x)=\sum_{j=1}^n c_j K(x,x_j)+\tilde p,\qquad c_j\in\RR,\quad \tilde p\in P
\end{equation}
can be used to solve the interpolation problem
\begin{equation}\label{intc}
\sigma(x_i)=f_i,\qquad i=1,\ldots,n.
\end{equation}
Moreover, a solution of (\ref{intc}) satisfying the condition
\begin{equation}\label{polyc}
\sum_{j=1}^n c_jp(x_j)=0\quad\hbox{for all } p\in P,
\end{equation}
can always be found \cite[p.~117]{Wendland}. This solution is unique if $X$ is a \emph{determining set} for $P$, that is
$p\in P$ and $p|_X=0$  implies $p=0$.

In meshless finite difference methods, conditionally positive definite kernels with respect to
spaces of polynomials are often used to produce numerical approximations of linear functionals
\begin{equation}\label{numdif}
\lambda f\approx \sum_{i=1}^n w_i f(x_i),\quad w_i\in\RR,
\end{equation}
such as the value $\lambda f=Df(x)$ of a differential operator $D$ applied to a function $f$ at a point 
$x\in\Omega$. 
If the interpolant $\sigma=\sigma_f$ satisfying (\ref{sigma})--(\ref{polyc}) with $f_i=f(x_i)$ is uniquely defined, then the weights $w_i$
of (\ref{numdif}) can be obtained by the approximation $\lambda f\approx\lambda\sigma_f$, which leads to the conditions
\begin{eqnarray}\label{kere}
\sum_{j=1}^n w_j K(x_i,x_j)+\tilde p(x_i)&=&\lambda'K(x_i),\quad i=1,\ldots,n,\quad\hbox{for some }\tilde p\in P,\\
\label{polye}
\sum_{i=1}^n w_i p(x_i)&=&\lambda p\quad\hbox{for all } p\in P,
\end{eqnarray}
where $\lambda'K:\Omega\time\Omega\to\RR$ is the function obtained by applying $\lambda$ to the first argument of $K$.
The weights $w_i$ are in this case uniquely determined by the conditions (\ref{kere})--(\ref{polye}). In particular, by introducing a basis for the space $P$, 
we can write  both (\ref{sigma})--(\ref{polyc}) and (\ref{kere})--(\ref{polye}) as systems of linear equations with the same matrix which is non-singular as soon as 
$X$ is a determining set for $P$. Solving this system is the standard way to obtain the weights $w_i$, 
see e.g.\ \cite{FFprimer15}. It is demonstrated in \cite{BFFB17} that  the weights satisfying (\ref{kere})--(\ref{polye})
for a polyharmonic kernel $K$ significantly improve the
performance of meshless finite difference methods in comparison to the weights obtained by
unconditionally positive definite kernels such as the Gaussian. In addition, these weights provide optimal recovery of
$\lambda f$ from the data $f(x_i)$, $i=1,\ldots,n$, on spaces of functions of finite smoothness, 
see e.g.~\cite[Chapter 13]{Wendland}. 

An alternative interpretation of (\ref{kere})--(\ref{polye}) is that the approximation (\ref{numdif}) of $\lambda f$
is required to be exact for all $f=\sigma$ in the form (\ref{sigma})
with coefficients $c_j$ satisfying (\ref{polyc}). Indeed, this can be easily shown
 with the help of the Fredholm alternative for matrices, see Theorem~\ref{exact} below.
In particular, (\ref{polye}) already expresses exactness of (\ref{numdif}) for all elements of $P$.
In the case when $\Omega$ is a domain in $\RR^d$ and $P$ is 
a space of $d$-variate polynomials, (\ref{polye}) can be used to obtain error bounds for the numerical differentiation 
with weights $w_i$, see e.g.\ \cite{DavySchaback18,DavySchaback19}.

However, exactness (\ref{polye}) for $p\in P$ is sometimes achievable without $X$ being a 
determining set for $P$. We then say that $X$ is \emph{$P$-consistent} for $\lambda$. The best known examples are the Gauss quadrature when $\lambda f=\int_a^bf(x)\,dx$ and the five point
stencil for the two-dimensional Laplacian. Moreover, $P$-consistent sets with $n$ significantly smaller than the 
dimension of $P$ often can be used for the numerical discretization of the Laplace operator on gridded nodes in irregular
domains, leading to sparser differentiation matrices \cite{D19arxiv}.

In this paper we study numerical approximation formulas (\ref{numdif}) obtained by requiring exactness conditions 
(\ref{kere})--(\ref{polye}) on ``deficient'' sets $X$ that are not determining for $P$. 

Our main result 
(Theorem 2 and Corollary 1) shows 
that a unique formula satisfying these conditions exists as soon as $X$ is $P$-consistent. Another consequence is that
the coefficients $c_j$ of the interpolant (\ref{sigma}) are uniquely defined for any $X$ (Corollary 2). Numerical differentiation
formulas obtained in this way  provide optimal recovery on native spaces of the kernels.  We also discuss computational methods for
the weights of the formula (\ref{numdif}) and coefficients of the interpolant (\ref{sigma}). In particular, a null space
method can be used for the saddle point problems  (\ref{kere})--(\ref{polye}) or (\ref{sigma})--(\ref{polyc}) even if in the
case of deficient sets they do not satisfy restrictions usually required in the literature \cite{BGL05}.

In the last section we describe two types of deficient sets that arise naturally in applications. First, deficient 
subsets of a grid may be used for numerical differentiation of the Laplacian (Section 3.1). Second, function values and
differential operators on algebraic surfaces, in this case an ellipse, may  be approximated using data  located on the
manifold, which are necessarily deficient sets for polynomials in the ambient space of degree at least the order of the 
surface (Section 3.2). In both cases, numerical results demonstrate a robust performance of the suggested numerical methods, 
and a reasonable approximation quality of the polyharmonic kernels we employ in the experiments.

\section{Approximation on deficient sets}
We assume that $K:\Omega\times\Omega\to\RR$ is a conditionally positive definite kernel with respect to a linear space $P$ of functions on $\Omega$, 
with $\dim P=m$. Let $\{p_1,\ldots,p_m\}$ be a basis for $P$. By writing $\tilde p=\sum_{j=1}^m v_j p_j$, $v_j\in\RR$,
 conditions (\ref{kere})--(\ref{polye}) give rise to a linear system with respect to $w_j$ and $v_j$, in block matrix form,
\begin{equation}\label{kpem}
\left[ \begin{array}{cc}
K_X & P_X  \\
P_X^T & 0 \end{array} \right] \cdot
\left[ \begin{array}{c}
w \\
v\end{array} \right]= 
\left[ \begin{array}{c}
a \\
b\end{array} \right], 
\end{equation}
where 
$$
K_X=[K(x_i,x_j)]_{i,j=1}^n,\quad P_X=[p_j(x_i)]_{i,j=1}^{n,m},$$
$$
w=[w_j]_{j=1}^n,\quad v=[v_j]_{j=1}^m,\quad
a=[\lambda'K(x_i)]_{i=1}^n,\quad b=[\lambda p_j]_{j=1}^{m}.$$
Condition (\ref{polyc}) in matrix form is
$$
P_X^Tc=0,\quad c=[c_j]_{j=1}^{n},$$
that is $c$ belongs to the null space $N(P_X^T)$ of $P_X^T$. Since 
$$
P_Xv=[\tilde p(x_i)]_{i=1}^{n},\quad \tilde p=\sum_{j=1}^m v_j p_j,$$
we see that the condition that $X$ is a determining set for $P$ is equivalent to
$N(P_X)=0$.

We  show that  the conditions (\ref{kere})--(\ref{polye}) %
express the exactness of (\ref{numdif}) for the sums $\sigma$ conditional on (\ref{polyc}), even when $X$ is 
a \emph{deficient set} for $P$, that is $N(P_X)\ne0$. Recall that this condition is equivalent to $R(P^T_X)\ne\RR^m$,
where $R(A)$ denotes the range of a matrix $A$.

\begin{theorem}\label{exact}   Let $X=\{x_1,\ldots,x_n\}\subset\Omega$. An approximation formula (\ref{numdif})  satisfies the exactness condition
$\lambda \sigma= \sum_{i=1}^n w_i \sigma(x_i)$ for all 
sums $\sigma$ in the form (\ref{sigma}) with coefficients $c_j$ satisfying (\ref{polyc}) if and only if (\ref{kere})--(\ref{polye}) holds
for the weights  $w_i$, $i=1,\ldots,n$. 
\end{theorem}

\begin{proof} The exactness condition is
$$
\sum_{j=1}^n c_j \lambda'K(x_j)+\lambda p= \sum_{j=1}^n c_j\sum_{i=1}^n w_i K(x_i,x_j)+\sum_{i=1}^n w_i p(x_i)$$
for all $c=[c_j]_{j=1}^{n}$ satisfying (\ref{polyc}) and all $p\in P$. In particular, for $c=0$ we obtain (\ref{polye}), and
rewrite the condition as
$$
\sum_{j=1}^n c_j \Big(\lambda'K(x_j)-\sum_{i=1}^n w_i K(x_i,x_j)\Big)=0\quad\hbox{for all } c\in N(P_X^T).$$
By the Fredholm alternative for matrices this is equivalent to
$$
\Big[\lambda'K(x_j)-\sum_{i=1}^n w_i K(x_i,x_j)\Big]_{j=1}^n\in R(P_X),$$
which is in turn equivalent to (\ref{kere}) in view of the symmetry of the
kernel $K$. \qed
\end{proof}

Linear systems of the type (\ref{kpem}) have been extensively studied under the name of equilibrium equations
\cite[Section 4.4.6]{GoVanL96} or saddle point problems \cite{BGL05} because they arise in many application areas.
Our approach below is a variation of the null space techniques described in \cite[Section 6]{BGL05}. However,
usual assumptions that $n\ge m$, $K_X$ is positive semidefinite and $P_X$ has full column rank are not satisfied 
in our case of interest. 

As long as $X$ is a deficient set,  $R(P^T_X)\ne\RR^m$ and hence the solvability of $P^T_Xw=b$  cannot be 
guaranteed for all $b$. Nevertheless, should this last equation
have a solution for $w$, there is a unique weight vector $w$ satisfying (\ref{kpem}).

\begin{theorem}\label{weights} There is a unique vector $w$ satisfying (\ref{kpem}) if and only if $b\in R(P^T_X)$.
\end{theorem}

\begin{proof} The necessity of the condition $b\in R(P^T_X)$ is obvious. To show the sufficiency, assume that $P^T_Xw_0=b$ for some $w_0\in\RR^n$. Then the solution $w$
must satisfy $P^T_X(w-w_0)=0$ if it exists, so we look for $w$ in the form
$$
w=w_0+\tilde u,\quad \tilde u\in N(P^T_X).$$
Let $M$ be a matrix whose columns form a basis for $N(P^T_X)$. Then $\tilde u=Mu$ for some vector $u$, 
and we may write (\ref{kpem}) equivalently as a linear system with respect to $u$ and $v$,
\begin{equation}\label{veq}
K_XMu+P_Xv=a-K_Xw_0.
\end{equation}
Since $M^TP_X=0$, it follows that
\begin{equation}\label{ueq}
M^TK_XMu=M^T(a-K_Xw_0).
\end{equation}
Since $K$ is conditionally positive definite, the matrix $M^TK_XM$ is positive definite, and hence there is a unique
$u$ determined by the last equation. The existence of some $v\in\RR^m$ such that
(\ref{veq}) holds is equivalent to the claim that $K_XMu-a+K_Xw_0\in R(P_X)$. This claim 
follows from  the Fredholm alternative since (\ref{ueq}) implies that $K_XMu-a+K_Xw_0\perp N(P^T_X)$. 
Thus, $u$ and $v$ satisfying (\ref{veq}) exist, and $u$ is uniquely determined. Then $w=w_0+Mu$ is a unique 
vector satisfying (\ref{kpem}). \qed
\end{proof}

\begin{remark}\label{weightsg}
Theorem~\ref{weights} is valid for any linear system (\ref{kpem}) with arbitrary matrices $A$ and $B$ replacing $K_X$ 
and $P_X$, respectively, and arbitrary $a,b$, as soon as $A$ is definite on $N(B^T)$, that is $x^TAx\ne0$ for all
$x\in N(B^T)\setminus\{0\}$. Indeed, this condition implies that $M^TAM$ is non-singular and hence the argument in the
proof goes through. %
\end{remark}

As long as the condition $b\in R(P^T_X)$ is satisfied, the weight vector $w$ may be found by any solution method applicable to the system
(\ref{kpem}), for example via the pseudoinverse of its matrix when it is singular. Alternatively, we may use the null space matrix $M$ 
of the above proof and find $w$ from the linear system 
\begin{equation}\label{kpem1}
\left[ \begin{array}{c}
M^TK_X  \\
P_X^T  \end{array} \right] 
w= 
\left[ \begin{array}{c}
M^Ta \\
b\end{array} \right], 
\end{equation}
which is in general overdetermined but
has full rank because its solution $w$ is unique. Indeed, any solution $w$ of (\ref{kpem1}) satisfied (\ref{kpem}) for some
$v$ since $M^TK_Xw=M^Ta$ implies $K_Xw-a\perp N(P^T_X)$ and thus $K_Xw-a\in R(P_X)$. We refer to 
 \cite[Section 6]{BGL05} for the computational methods for the null space matrix $M$. 
 One obvious possibility is to employ the right singular vectors of $P_X$, see 
 \cite[Eq.\ (2.5.4)]{GoVanL96}. Should $v$ be needed, it can be computed
as a solution of the consistent linear system 
\begin{equation}\label{vo}
P_Xv=a-K_Xw. 
\end{equation}
For example we can use
\begin{equation}\label{v}
v=P_X^+(a-K_Xw),
\end{equation} 
where $P_X^+$ denotes the Moore-Penrose pseudoinverse
of $P_X$, is the unique $v$ with the smallest 2-norm.

We formulate two immediate consequences of Theorem~\ref{weights} for the numerical approximation of functionals and for the interpolation. Note that 
(\ref{sigma})--(\ref{polyc}) can be written in the form (\ref{kpem}) with $w$ replaced by $c$, $a=[f_i]_{i=1}^n$, and $b=0$. In particular, the condition
$0\in R(P^T_X)$ of Theorem~\ref{weights} is trivially satisfied.

\begin{corollary}\label{ndw} 
For any $X$ and $\lambda$ there is a unique numerical approximation formula (\ref{numdif}) satisfying 
(\ref{kere})--(\ref{polye}) as soon as (\ref{polye}) is solvable. 
\end{corollary}

\begin{corollary}\label{intw} 
For any data $(x_j,f_j)$, $j=1,\ldots,n$, one or more interpolants $\sigma$ satisfying (\ref{sigma})--(\ref{polyc}) exist and their coefficients $c_j$,
$j=1,\ldots,n$, are uniquely determined.
\end{corollary}

Thanks to Theorem~\ref{exact} we also obtain the property known for the case of a determining set $X$ that the approximation 
$\lambda f\approx\sum_{i=1}^n w_i f(x_i)$ can be found by requiring $\lambda f\approx\lambda\sigma$ for any interpolant
$\sigma$ of Corollary~\ref{intw} with $f_i=f(x_i)$.

Looking specifically at numerical differentiation, consider the case when $\Omega=\RR^d$, $\lambda f=Df(x)$ for a linear
differential operator $D$ of order $k$ and $x\in\RR^d$, and $P=\PP^d_q$, the space of $d$-variate polynomials of total order at most $q$
(that is, total degree at most $q-1$) for some $q\in\NN$. Any kernel $K$ that is conditionally positive definite with respect to $P$
generates a native semi-Hilbert space $F(K,P)$ of functions on $\Omega$ with null space $P$, see e.g.~\cite{Wendland}. By inspecting the
arguments in Section~2 and Lemma~6 of \cite{DavySchaback16}, we see that thanks to Corollary~\ref{ndw}, the optimal recovery property of the weights $w_i$ defined by 
(\ref{kere})--(\ref{polye}) remains valid for deficient sets $X=\{x_1,\ldots,x_n\}$. 
More precisely, %
the worst case error of numerical differentiation formulas on the unit ball of $F_q(K):=F(K,\PP^d_q)$,
$$
E(u):=\sup_{f\in F_q(K)\atop \|f\|_{F_q(K)}\le 1}\big|Df(x)-\sum_{i=1}^n u_i f(x_i)\big|,$$
can be computed as
\begin{equation}\label{wcerror}
\begin{array}{rcl}
  E^2(u)&=& \displaystyle D'D''K(x,x)-\sum_{i=1}^n u_i\big(D'K(x,x_i)+D''K(x_i,x)\big)\\[4pt]
  &&\displaystyle +\sum_{i,j=1}^n u_iu_jK(x_i,x_j).
\end{array}
\end{equation}
The weight vector $w$ has the \emph{optimal recovery} property in the sense that it  satisfies
\begin{equation}\label{optrec}
E(w) =\min\big\{E(u):u\in\RR^n,\; Dp(x)=\sum_{i=1}^n u_i p(x_i)\;\hbox{ for all }p\in \PP^d_q\big\}
\end{equation}
as soon as the mixed partial derivatives of $K$ exist at $(x,x)\in\RR^d\times\RR^d$ up to the order $k$ in each of both
$d$-dimensional variables, and
$X$ is such that there exists a vector $u\in\RR^n$ with polynomial exactness
$$
Dp(x)=\sum_{i=1}^n u_i p(x_i)\;\hbox{ for all }p\in \PP^d_q.$$
We use here $D'$ and $D''$ to indicate when $\lambda f=Df(x)$ acts on the first, respectively, the second argument of $K$.

Note that the equality-constraned quadratic minimization problem (\ref{optrec}) provides an alternative way of computing
the optimal weight vector on a deficient set. %
By Theorem~\ref{weights} we know that its solution $w$
is unique as soon as the feasible region is non-empty.

\section{Examples}
In this section we illustrate Corollaries \ref{ndw} and \ref{intw} on particular examples where deficient sets $X$ seem
useful.

We consider the \emph{polyharmonic kernels} $K_{s,d}:\RR^d\times \RR^d\to\RR$, defined for all real
$s>0$ by $K_{s,d}(x,y)=\varphi_{s}(\|x-y\|_2)$, where
\begin{equation}\label{polyh}
\varphi_{s}(r):=(-1)^{\lfloor s/2 \rfloor +1}
\left\{
\begin{array}{ll}
r^s\log r, &\hbox{ if $s$ is an even integer, }\\ 
r^s,& \hbox{ otherwise. }
\end{array} 
\right.
\end{equation}
The kernel $K_{s,d}$ is conditionally positive definite with respect to $\PP^d_q$ for all
$q\ge\lfloor s/2 \rfloor +1$. We cite \cite{Wendland,DavySchaback19} and references therein 
for further information on these kernels.
If $m=(s+d)/2$ is an integer and  $q$ is chosen equal to $m$, then
the native space $F_m(K_{s,d})$ coincides with the Beppo-Levi space $BL_{m}(\RR^d)$, see
\cite[Theorem 10.43]{Wendland}. For any $q\ge \lfloor s/2 \rfloor +1$, the space $F_q(K_{s,d})$ can be described 
with the help of the generalized Fourier transforms as in \cite[Theorem 10.21]{Wendland}. 
By the arguments in Section~2, formulas (\ref{wcerror}) and (\ref{optrec}) apply to  $K_{s,d}$ as soon as $s>2k$, where $k$ is
the order of the differential operator $D$.

\subsection{Numerical differentiation of Laplacian on a grid}\label{Lgrid}
We are looking for numerical differentiation formulas of the type
\begin{equation}\label{ndL}
\Delta f(0)\approx \sum_{\alpha\in Z_{d,r}}w_\alpha f(\alpha),
\quad Z_{d,r}:=\{\alpha\in\ZZ^d:\|\alpha\|_2\le r\},\quad r>0,
\end{equation}
where $\Delta$ is the Laplacian
$\Delta f=\sum_{i=1}^d{\partial^2 f}/{\partial x_i^2}$. 
The set $Z_{d,r}$ for $0\le r<1$ consists of the origin only and hence is not useful for the approximation of the Laplacian.
For $r=1$ we have
$Z_{d,1}=\{0,\pm e_1,\ldots \pm e_d\}$,
where $e_i$ is the $i$-th unit vector in $\RR^d$,
and
\begin{equation}\label{star}
\Delta f(0)\approx -2df(0)+\sum_{i=1}^df(e_i)+\sum_{i=1}^df(-e_i)
\end{equation}
is the classical numerical differentiation formula exact for all cubic polynomials $f=p\in\PP^d_4$. Hence 
(\ref{polye}) is solvable for all $X=Z_{d,r}$, $r\ge1$, if $P=\PP^d_4$.

According to Corollary \ref{ndw}, we have computed the unique weights of the formula (\ref{ndL}) 
satisfying (\ref{kere})--(\ref{polye})
for the kernel $K_{7,d}$, $P=\PP^d_4$ and $X=Z_{d,r}$
for all $d=2,\ldots,5$ and $r=1,\sqrt{2},\sqrt{3},2$. As a basis for $\PP^d_4$ we choose ordinary monomials. However,
the computation is performed using the rescaling of $X$ as $X/r$ according to the suggestion in \cite[Section 6.1]{DavySchaback18}

\begin{table}[h!]%
\label{Lgridt}
\caption{Numerical differentiation of Laplacian on a grid: $|X|$ is the cardinality of $X$, $dN=\dim N(P_X)$, $dNt=\dim N(P_X^T)$, $E(w)$
is given by (\ref{wcerror}), and $cond$ is the condition number of the matrix of 
(\ref{kpem1}).} 
\begin{center}
\begin{tabular}{|c@{\enskip}|@{\enskip}r@{\enskip}|@{\enskip}r@{\enskip}|%
                @{\enskip}r@{\enskip}|@{\enskip}r@{\enskip}|@{\enskip}r@{\enskip}|@{\enskip}c@{\enskip}|}%
\hline
\multicolumn{7}{|c|}{\rule{0pt}{12pt}$d=2$, $X=Z_{2,r}$, $\dim \PP^2_4=10$}\\[2pt]
\hline\rule{0pt}{12pt}
$r$ & $|X|$ & $dN$ & $dNt$ & $E(w)$ & $\|w\|_1$ & $cond$ \\[2pt]
\hline\rule{0pt}{12pt}
1          &  5 &  5 & 0 & 13.4 & 8.0   & - \\
$\sqrt{2}$ &  9 &  2 & 1 & 10.6 & 13.5  & 2.0e\spl02 \\
$\sqrt{3}$ &  9 &  2 & 1 & 10.6 & 13.5  & 2.0e\spl02  \\
2          & 13 &  0 & 3 & 7.4  & 11.8  & 3.9e\spl02 \\[2pt]
\hline
\hline
\multicolumn{7}{|c|}{\rule{0pt}{12pt}$d=3$, $X=Z_{3,r}$, $\dim \PP^3_4=20$}\\[2pt]
\hline\rule{0pt}{12pt}
$r$ & $|X|$ & $dN$ & $dNt$ & $E(w)$ & $\|w\|_1$ & $cond$ \\[2pt]
\hline\rule{0pt}{12pt}
1  & 7  &  13   &  0  &   17.2    &    12.0 & - \\
$\sqrt{2}$ & 19   &   4    &  3    &  12.3   &  22.7   &   3.8e\spl02   \\
$\sqrt{3}$ & 27   &  3  &  10  &   12.4  &  24.8  &  2.5e\spl03   \\
2          & 33    &  0   &  13   &   9.0   &   30.1   &  5.1e\spl03 \\[2pt]
\hline
\hline
\multicolumn{7}{|c|}{\rule{0pt}{12pt}$d=4$, $X=Z_{4,r}$, $\dim \PP^4_4=35$}\\[2pt]
\hline\rule{0pt}{12pt}
$r$ & $|X|$ & $dN$ & $dNt$ & $E(w)$ & $\|w\|_1$ & $cond$ \\[2pt]
\hline\rule{0pt}{12pt}
1  & 9   &   26   &    0   &    20.8     &     16.0  & - \\
$\sqrt{2}$ & 33   &  8  &   6  &    14.0  &  31.8  &  5.7e\spl02  \\
$\sqrt{3}$ &  65  &   4  &  34  &   13.9 &  39.7  &  6.9e\spl03  \\
2          & 89  &   0  &  54   &  10.4  &  40.5  &   3.1e\spl04  \\[2pt]
\hline
\hline
\multicolumn{7}{|c|}{\rule{0pt}{12pt}$d=5$, $X=Z_{5,r}$, $\dim \PP^5_4=56$}\\[2pt]
\hline\rule{0pt}{12pt}
$r$ & $|X|$ & $dN$ & $dNt$ & $E(w)$ & $\|w\|_1$ & $cond$ \\[2pt]
\hline\rule{0pt}{12pt}
1  & 11  &   45  &     0 &    24.2   &      20.0 & - \\
$\sqrt{2}$ & 51  &  15   &  10  &  15.6  &   40.9  &   7.7e\spl02  \\
$\sqrt{3}$ & 131  &    5  &    80  &   15.4  &   56.4  &    1.3e\spl04   \\
2          & 221  &   0  &  165  &  11.7  &  55.0  &   9.0e\spl04  \\[2pt]
\hline
\end{tabular}
\end{center}
\end{table}

Table~1 presents information about the 
size $|X|$ of $X$, dimensions of the null spaces of $P_X$ and $P_X^T$,  the optimal recovery error (\ref{optrec})
on $F_4(K_{7,d})$, the stability constant of the weight vector $\|w\|_1=\sum_{i=1}^n|w_i|$, and the condition number $cond$
of the system (\ref{kpem1}) we solved in order to compute the weights for $r\ne1$. A smaller optimal recovery error indicates
better approximation quality, whereas $\|w\|_1$ and $cond$ measure the numerical stability of the formulas.
Note that $\dim N(P_X^T)=0$ for $r=1$, which means that (\ref{star}) is the only
solution of (\ref{polye}) in this case, and hence it provides the optimal recovery on $F_4(K_{7,d})$. 
For $r=2$ we have $\dim N(P_X)=0$ and it follows that $Z_{d,2}$ is a determining set for $\PP^d_4$. For $r=\sqrt{2},\sqrt{3}$ we obtain
examples of optimal recovery weights on deficient sets, with $\dim N(P_X)$ being the dimension
of the affine space of weight vectors satisfying the polynomial exactness condition (\ref{polye}). 
These new weights seem to provide a meaningful choice for the two intermediate sets 
between the classical polynomial stencil on  $Z_{d,1}$, and the standard polyharmonic weights on the determining set 
$Z_{d,2}$. Indeed, as expected, the optimal recovery error $E(w)$ reduces when $|X|$ increases, whereas the stability 
constant and condition numbers tend to increase. 

\subsection{Interpolation of data on ellipse}\label{ellipse}

In this example we compute the kernel interpolant (\ref{sigma}) satisfying (\ref{polyc}) and $\sigma(x_i)=f(x_i)$, $i=1,\ldots,n$, 
for the test function $f:\RR^2\to\RR$ given by
$$
f(x,y)=\sin(\pi x)\sin(\pi y).$$
We use the polyharmonic kernels $K_{s,2}$ and $P=\PP^2_q$ for the pairs
$$
(s,q)=(5,3),(7,4),(9,5),$$
and choose sets $X$ with $n=|X|=5\cdot 2^i$, $i=0,1,\ldots 6$, on the ellipse $\cal E$ with half-axes $a=1$ and $b=0.75$ centered at the origin.
The sets are obtained by first choosing parameter values $t_i=ih$, $i=0,\ldots,n-1$, where $h=2\pi/n$, then adding to each
$t_i$ a random number $\epsilon_i$ with uniform distribution in the interval $[-0.3h,0.3h]$, and selecting 
$x_i=\big(a\cos(t_i+\epsilon_i),b\sin(t_i+\epsilon_i)\big)$. The first two sets used in our experiments are shown in Fig.~1.

\begin{figure}[htbp!]
\begin{center}
\hspace*{-10pt}\includegraphics[width=120pt]{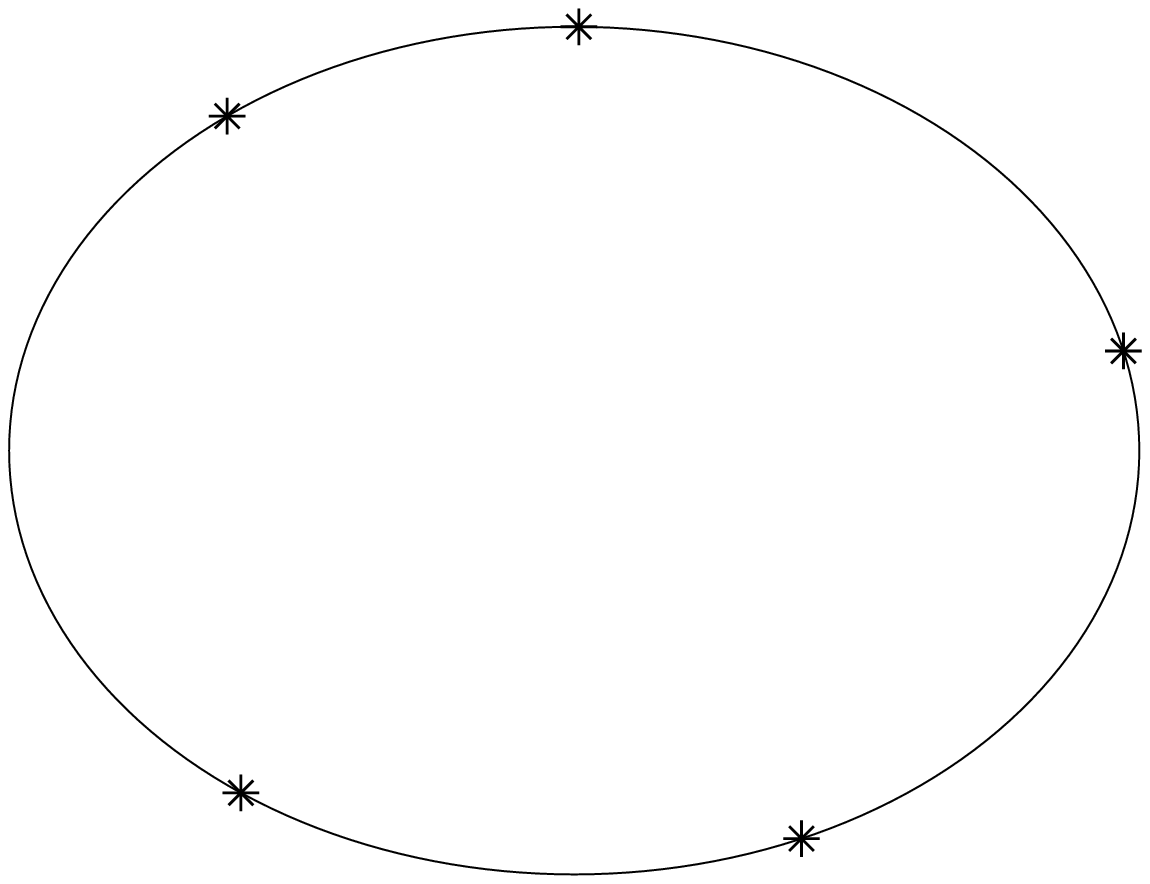}\qquad\qquad\includegraphics[width=120pt]{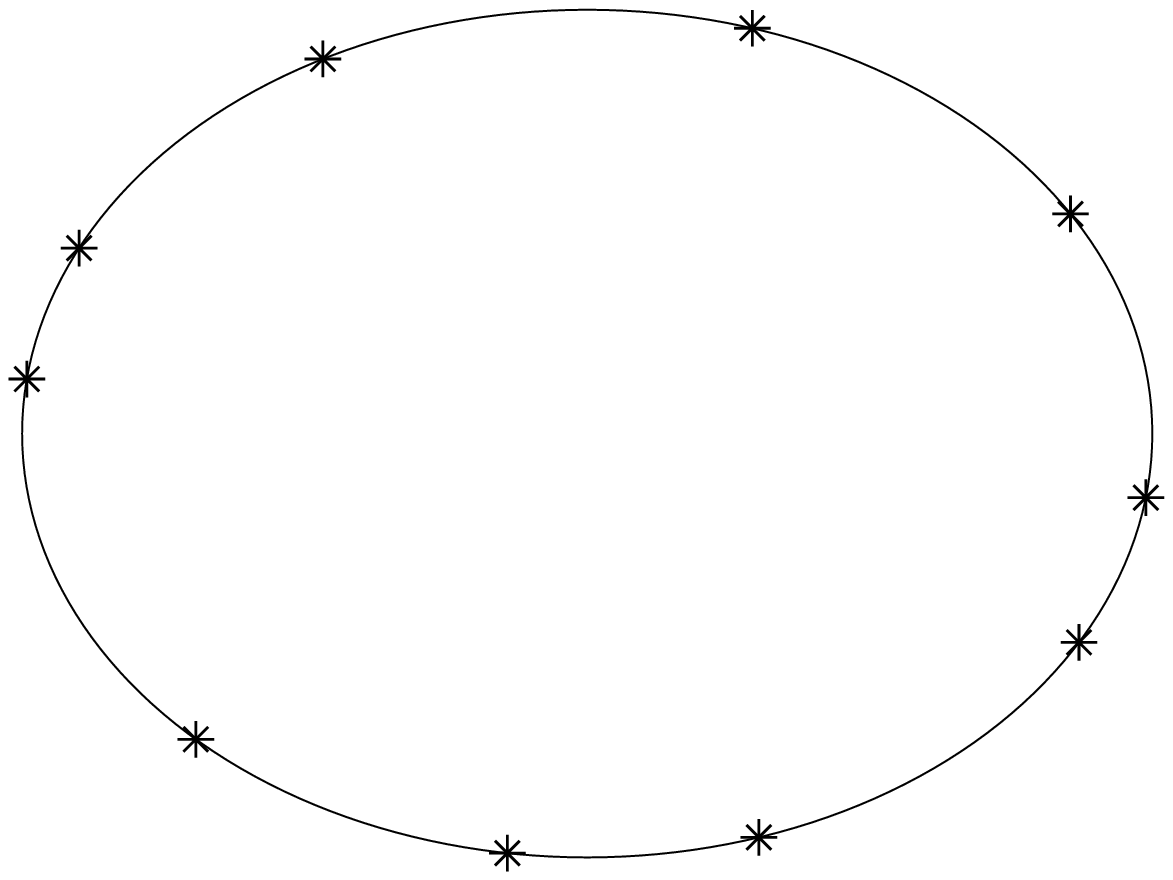}
\end{center}
\caption{Interpolation of data on ellipse: The sets $X$ with $|X|=5$ (left) and 10 (right).}
\end{figure}

Since $X\subset\cal E$ and there exists a nontrivial quadratic polynomial $p\in \PP^2_3$
that vanishes on $\cal E$, 
$$
p(x,y)=x^2/a^2+y^2/b^2-1,$$
all sets $X$ are deficient for $\PP^2_q$, $q\ge3$. Nevertheless, 
according to Corollary~\ref{intw}, the coefficients $c_j$ of the interpolant $\sigma$ in (\ref{sigma}) 
are uniquely determined and can be computed by solving the system (\ref{kpem1}). A polynomial 
$\tilde p$ of (\ref{sigma}), $\tilde p=v_1p_1+\cdots +v_m p_m$, can be computed by solving (\ref{vo}). We will use the
pseudoinverse as in (\ref{v}), but in fact the polynomial $\tilde p$ is uniquely determined on the ellipse $\cal E$ as soon as
$n\ge 2q-1$. Indeed, if both $\tilde p_1,\tilde p_2\in\PP^2_q$ satisfy (\ref{vo}), and 
$\tilde p_1-\tilde p_2=u_1p_1+\cdots +u_m p_m$, then $P_Xu=0$, which implies $(\tilde p_1-\tilde p_2)|_X=0$. Hence
$x_1,\ldots,x_n$ are intersection points of the ellipse and the zero curve of $\tilde p_1-\tilde p_2$, an algebraic curve of order 
$q-1$. By Bezout theorem, this curve must contain $\cal E$ as soon as $n>2(q-1)$, which implies
$\tilde p_1|_{\cal E}=\tilde p_2|_{\cal E}$. 

Thus, $\sigma|_{\cal E}$ is well defined as soon as $|X|\ge5$ for $q=3$, 
$|X|\ge7$ for $q=4$ and $|X|\ge9$ for $q=5$. We are using $\sigma(x)$ as an approximation of $f(x)$ for $x\in\cal E$.
Moreover, we also approximate the surface gradient 
$$
\nabla_{\cal E}f(x):=\nabla f(x)-\nabla f(x)^T\nu(x)\cdot\nu(x),\quad x\in \cal E,$$ 
where $\nu(x)$ is the unit outer normal to $\cal E$ at $x$. The surface gradient $\nabla_{\cal E}f(x)$ can either be
approximated by $\nabla_{\cal E}\sigma(x)$, or by using a numerical differentiation formula (\ref{numdif}), 
with the same result. For each $X$, except of $|X|=5$ for $q=4,5$, we evaluated the maximum error of the function
and surface gradient, 
\begin{eqnarray*}
max&=&\max_{x\in\cal E}|f(x)-\sigma(x)|,\\
maxg&=&\max_{x\in\cal E}\|\nabla_{\cal E}f(x)-\nabla_{\cal E}\sigma(x)\|_2,
\end{eqnarray*}
by sampling the parameter $t$ of the ellipse $\big(a\cos t,b\sin t\big)$, $t\in[0,2\pi)$, equidistantly 
with the step $h/20$. The results are presented in Table 2, where we also included
the condition number $cond$ of the system (\ref{kpem1}). Note that we translate and scale $X$ using its center of gravity 
$z$, and perform the computations with $K_{s,2}$ and ordinary monomials on the set 
$Y=(X-z)/\max\{\|x_i-z\|_2:i=1,\ldots,n\}$, in order to improve the condition numbers.
The results in the table demonstrate a fast convergence of the interpolant $\sigma$ and its surface gradient to $f$ and 
$\nabla_{\cal E}f$. Note that although the condition numbers become high when the set $X$ fills the ellipse more densely,
they are moderate in comparison to significantly worse conditioned matrices arising if 
infinitely smooth kernels such as the Gaussian $K_{G,\varepsilon}(x-y)=\exp(-\varepsilon\|x-y\|_2^2)$, $\varepsilon>0$, 
are employed, see also discussions in \cite[Section 5.1.5]{FFprimer15}.

\begin{table}[htbp!]
\label{ellipset}
\caption{Interpolation of a test function on an ellipse using the kernel $K_{s,2}$ and $P=\PP^2_q$: $|X|$ is the cardinality of $X$, 
$max$ and $maxg$ are the maximum errors of the function or  the surface gradient, respectively, and $cond$ is the condition number of the matrix of 
(\ref{kpem1}).} 
\begin{center}
\begin{tabular}{|r@{\enskip}|@{\hspace{3pt}}c@{\hspace{3pt}}|@{\hspace{3pt}}c@{\hspace{3pt}}|@{\hspace{3pt}}c@{\hspace{3pt}}|%
                @{\hspace{3pt}}c@{\hspace{3pt}}|@{\hspace{3pt}}c@{\hspace{3pt}}|@{\hspace{3pt}}c@{\hspace{3pt}}|%
                @{\hspace{3pt}}c@{\hspace{3pt}}|@{\hspace{3pt}}c@{\hspace{3pt}}|@{\hspace{3pt}}c@{\hspace{3pt}}|}%
\hline
&\multicolumn{3}{c|}{\rule{0pt}{12pt}$s=5$, $q=3$}&\multicolumn{3}{c|}{\rule{0pt}{12pt}$s=7$, $q=4$}
&\multicolumn{3}{c|}{\rule{0pt}{12pt}$s=9$, $q=5$}\\[2pt]
\hline\rule{0pt}{12pt}
$|X|$ & $max$ & $maxg$  & $cond$ & $max$ & $maxg$  & $cond$ & $max$ & $maxg$  & $cond$\\[2pt]
\hline\rule{0pt}{12pt}
5  &  8.3e\sml01   &   2.5e\spl00  &    4.7e\spl00 & -    &        -   &  - & -    &   -  &   -\\
10 & 2.5e\sml01  &  9.9e\sml01  &  3.0e\spl02 & 1.8e\sml01  &   7.9e\sml01  &   1.9e\spl02 & 1.6e\sml01  &   6.9e\sml01 &    2.3e\spl02\\
20 & 2.9e\sml03  &  2.0e\sml02  &  2.3e\spl04 & 1.6e\sml03  &   1.1e\sml02  &   9.8e\spl04 & 1.6e\sml03  &   1.5e\sml02  &   2.0e\spl05\\
40 & 4.6e\sml05  &  6.9e\sml04  &  2.5e\spl06 & 2.2e\sml06  &   3.6e\sml05  &   5.3e\spl07 & 6.1e\sml07  &   1.4e\sml05  &   6.3e\spl08\\
80 &  1.0e\sml06 &   3.1e\sml05 &   2.4e\spl08 & 1.4e\sml08  &   3.5e\sml07  &   2.3e\spl10 & 4.2e\sml10  &   1.5e\sml08  &   1.2e\spl12\\
160 &  2.2e\sml08  &  1.2e\sml06  &  1.2e\spl10 & 7.5e\sml11  &   4.2e\sml09  &   4.8e\spl12 & 8.1e\sml12 &    1.8e\sml09  &   4.6e\spl15\\
320 &   2.7e\sml10 &   3.1e\sml08  &  9.2e\spl11 &  1.0e\sml12 &    2.2e\sml10 &   2.9e\spl15 & 9.1e\sml12 &   4.9e\sml10 &   5.8e\spl17\\[2pt]
\hline
\end{tabular}
\end{center}
\end{table}

\bibliographystyle{abbrv} %
\bibliography{bibl}

\begin{thebibliography}{1}

\bibitem{BFFB17}
V.~Bayona, N.~Flyer, B.~Fornberg, and G.~A. Barnett.
\newblock On the role of polynomials in {RBF-FD} approximations: {II}.
  {N}umerical solution of elliptic {PDE}s.
\newblock {\em Journal of Computational Physics}, 332:257 -- 273, 2017.

\bibitem{BGL05}
M.~Benzi, G.~H. Golub, and J.~Liesen.
\newblock Numerical solution of saddle point problems.
\newblock {\em Acta Numerica}, 14:1--137, 2005.

\bibitem{D19arxiv}
O.~Davydov.
\newblock Selection of sparse sets of influence for meshless finite difference
  methods, arxiv:1908.01567, 2019.

\bibitem{DavySchaback16}
O.~Davydov and R.~Schaback.
\newblock Error bounds for kernel-based numerical differentiation.
\newblock {\em Numerische Mathematik}, 132(2):243--269, 2016.

\bibitem{DavySchaback18}
O.~Davydov and R.~Schaback.
\newblock Minimal numerical differentiation formulas.
\newblock {\em Numerische Mathematik}, 140(3):555--592, 2018.

\bibitem{DavySchaback19}
O.~Davydov and R.~Schaback.
\newblock {Optimal stencils in {S}obolev spaces}.
\newblock {\em IMA Journal of Numerical Analysis}, 39(1):398--422, 2019.

\bibitem{FFprimer15}
B.~Fornberg and N.~Flyer.
\newblock {\em A Primer on Radial Basis Functions with Applications to the
  Geosciences}.
\newblock Society for Industrial and Applied Mathematics, Philadelphia, PA,
  USA, 2015.

\bibitem{GoVanL96}
G.~H. Golub and C.~F. Van~Loan.
\newblock {\em Matrix Computations}.
\newblock The Johns Hopkins University Press, third edition, 1996.

\bibitem{Wendland}
H.~Wendland.
\newblock {\em Scattered Data Approximation}.
\newblock Cambridge University Press, 2005.

\end{thebibliography}

\end{document}